\documentclass[preprint,dvipdfmx,leqno]{elsarticle}
\usepackage{
amsmath,amssymb,amsthm,ascmac,
bm,natbib,hyperref,here,xcolor
,txfonts
}
\usepackage{tikz}
\usetikzlibrary{intersections,calc}
\newtheoremstyle{mydefinition}
{}{}
{\normalfont}{}
{\bfseries}{}
{\newline}{\thmname{#1}\thmnumber{\ #2}\thmnote{\quad#3}.}
\newtheoremstyle{mydefinition2}
{}{}
{\normalfont}{}
{\bfseries}{}
{\newline}{\thmname{#1}\thmnumber{\ #2}\thmnote{\quad#3}}
\theoremstyle{mydefinition}
\newtheorem{Thm}{Theorem}
\newtheorem{Prop}[Thm]{Proposition}
\newtheorem{Lem}[Thm]{Lemma}
\theoremstyle{mydefinition2}
\newtheorem{Lem2}[Thm]{Lemma}
\newtheorem*{Lem*}{Lemma}
\makeatletter
\newcommand{\opnorm}{\@ifstar\@opnorms\@opnorm}
\newcommand{\@opnorms}[1]{
\left|\mkern-1.5mu\left|\mkern-1.5mu\left|#1
\right|\mkern-1.5mu\right|\mkern-1.5mu\right|}
\newcommand{\@opnorm}[2][]{
\mathopen{#1|\mkern-1.5mu#1|\mkern-1.5mu#1|}#2
\mathclose{#1|\mkern-1.5mu#1|\mkern-1.5mu#1|}}
{\end{smallmatrix}\right)}
{\end{smallmatrix}\right]}
{\end{smallmatrix}\right\}}
\@addtoreset{equation}{section}
\makeatother
\begin{document}
\newcounter{tmpcnt}
\setcounter{tmpcnt}{1}
\journal{arXiv}
\begin{frontmatter}
\title{Local well-posedness of the complex Ginzburg-Landau Equation in general domains}
\author[add1]{Takanori Kuroda}
\ead{1d\_est\_quod\_est@ruri.waseda.jp}
\author[add2,fn2]{Mitsuharu \^Otani}
\ead{otani@waseda.jp}
\fntext[fn2]{Partly supported by the Grant-in-Aid for Scientific Research, \#18K03382, the Ministry of Education, Culture, Sports, Science and Technology, Japan.}
\address[add1]{Department of Mathematics, School of Science and Engineering, \\ Waseda University, 3-4-1 Okubo Shinjuku-ku, Tokyo, 169-8555, JAPAN}
\address[add2]{Department of Applied Physics, School of Science and Engineering, \\ Waseda University, 3-4-1 Okubo Shinjuku-ku, Tokyo, 169-8555, JAPAN}
\begin{abstract}
In this paper, complex Ginzburg-Landau (CGL) equations with superlinear growth terms are studied.
We discuss the local well-posedness in the energy space \({\rm H}^1\) for the initial-boundary value problem of the equations in general domains.
The local well-posedness in \({\rm H}^1\) in bounded domains is already examined by authors (2019).
Our approach to CGL equations is based on the theory of parabolic equations governed by subdifferential operators with non-monotone perturbations.
By using this method together with the Yosida approximation procedure, we discuss the existence and the uniqueness of local solutions as well as the global existence of solutions with small initial data.
\end{abstract}
\end{frontmatter}
\section{Introduction}\label{sec-1}
In this paper we are concerned with the following complex Ginzburg-Landau equation in a general domain \(\Omega \subset \mathbb{R}^N\) with smooth boundary \(\partial \Omega\):\vspace{-1mm}
\begin{equation}
\tag*{(CGL)}
\left\{
\begin{aligned}
&\partial_t u(t, x) \!-\! (\lambda \!+\! i\alpha)\Delta u \!-\! (\kappa \!+\! i\beta)|u|^{q-2}u \!-\! \gamma u \!=\! f(t, x)
&&\hspace{-2mm}\mbox{in}\ (t, x) \in [0, T] \times \Omega,\\
&u(t, x) = 0&&\hspace{-2mm}\mbox{on}\ (t, x) \in [0, T] \times \partial\Omega,\\
&u(0, x) = u_0(x)&&\hspace{-2mm}\mbox{in}\ x \in \Omega,
\end{aligned}
\right.
\end{equation}
where \(\lambda, \kappa > 0\), \(\alpha, \beta, \gamma \in \mathbb{R}\) are parameters; 
\(i = \sqrt{-1}\) denotes the imaginary unit; 
\(u_0: \Omega \rightarrow \mathbb{C}\) is a given initial value; 
\(f: \Omega \times [0, T] \rightarrow \mathbb{C}\) (\(T > 0\)) is a given external force.
Our unknown function \(u:\overline{\Omega} \times [0,\infty) \rightarrow \mathbb{C}\) is complex valued.
Under the suitable assumption on \(q\), that is, the Sobolev subcritical condition on \(q\), we establish the local well-posedness of (CGL) for \(u_0 \in {\rm H}^1\).

As extreme cases, (CGL) corresponds to two well-known equations: 
semi-linear heat equations (when \(\alpha=\beta=0\)) and nonlinear Schr\"odinger equations (when \(\lambda=\kappa=0\)).
Thus for the general case, (CGL) could be regarded as an ``intermediate'' equation between these two equations.

As for the case where \(\kappa < 0\), equation (CGL) was introduced by Landau and Ginzburg in 1950 \cite{GL1} as a mathematical model for superconductivity.
Subsequently, it was revealed that many nonlinear partial differential equations arising from physics can be rewritten in the form of (CGL) (\cite{N1}).

Mathematical studies for the case where \(\kappa < 0\) are pursued extensively by several authors.
The first treatment is due to Temam \cite{T1}, where weak global solutions were constructed by the Galerkin method.
Ginibre-Velo \cite{GinibreJVeloG1996} showed the existence of global strong solutions for (CGL) in whole space \(\mathbb{R}^N\) under some suitable conditions on \(\lambda,\kappa,\alpha,\beta\) in terms of \(q\) with initial data taken form \(\mathrm{H}^1(\mathbb{R}^N)\cap\mathrm{L}^q(\mathbb{R}^N)\).
They approximated nonlinear terms by the mollifier and then established a priori estimates to obtain solutions of (CGL).
Okazawa-Yokota \cite{OY1} regarded (CGL) as parabolic equations with perturbations governed by maximal monotone operators in complex Hilbert spaces and showed the existence of global solutions together with some smoothing effects.
The global existence of solutions of (CGL) in general (unbounded) domains is investigated in \cite{KOS1}, where (CGL) is regarded as parabolic equations with monotone and non-monotone perturbations governed by subdifferential operators in the product space of real Hilbert spaces.

For the case where \(\kappa > 0\), the blow-up phenomenon of solutions was first shown by Masmoudi-Zaag \cite{MZ1} and Cazenave et al. \cite{CDW1,CDF1}.
As for the local well-posedness, Shimotsuma-Yokota-Yoshii \cite{SYY1} discussed in \({\rm L}^p\) for various kinds of domains via a suitable estimate for the heat kernel; the local well-posedness in \({\rm H}_0^1\) in bounded domains is studied in \cite{KO2} by the non-monotone perturbation theory of nonlinear parabolic equations governed by subdifferential operators developed in \cite{O2}.

In this paper, we show the local well-posedness in the energy space \({\rm H}_0^1\) for general domains.
We first regard (CGL) as an abstract evolution equation governed by subdifferential operators in the product Hilbert space \(({\rm L}^2(\Omega))^2\) over real coefficients.
Our previous work \cite{KO2} essentially relies on the compactness argument, which is guaranteed by boundedness of domains with the aid of Rellich-Kondrachov's theorem.
The heat kernel for \(-(\lambda+i\alpha)\Delta\) is constructed for general domains and examined in \cite{SYY1} for \({\rm L}^p\) (\(1<p<\infty\)) spaces.
However, estimates for derivatives of heat kernels for this elliptic operator with complex coefficients in general domains are not obtained.
Hence we can not apply this approach.
Here we propose a different strategy, i.e., we first introduce an auxiliary equation, which is (CGL) added with a dissipative term \(\varepsilon|u|^{r-2}u\) (\(r>q,\varepsilon>0\)) to dominate non-monotone terms \(-(\kappa+i\beta)|u|^{q-2}u\).
To show the global well-posedness of solutions for this auxiliary equation, we make use of the Yosida approximation instead of the compactness argument.
By establishing a priori estimates of solutions \(u_\varepsilon\) of auxiliary equations independent of \(\varepsilon>0\) and letting \(\varepsilon\) tend to \(0\), we show that \(u_\varepsilon\) converges to our desired solution.

In addition to the local existence result, we also give related results concerning global solutions.

This  paper consists of six sections.
In \S2, we fix some notations and prepare some preliminaries.
Main results are stated in \S3.
In \S4, we introduce auxiliary problems for (CGL) and show their global well-posedness.
The local well-posedness of (CGL) is discussed in \S5 as well as the alternative concerning the asymptotic behavior of solutions, i.e., blow-up or global existence of solutions.
The last section \S6 is devoted to the study of the existence of small global solutions.
\section{Notations and Preliminaries}\label{sec-2}
In this section, we fix some notations in order to formulate (CGL) as an evolution equation in a real product function space based on the following identification:
\[
\mathbb{C} \ni u_1 + iu_2 \mapsto (u_1, u_2)^{\rm T} \in \mathbb{R}^2.
\]
Then define the following:
\[
\begin{aligned}
&(U \cdot V)_{\mathbb{R}^2} :=  u_1 v_1 + u_2 v_2,\quad |U|=|U|_{\mathbb{R}^2}, \qquad U=(u_1, u_2)^{\rm T}, \ V=(v_1, v_2)^{\rm T} \in \mathbb{R}^2,\\[1mm]
&\mathbb{L}^2(\Omega) :={\rm L}^2(\Omega) \times {\rm L}^2(\Omega),\quad (U, V)_{\mathbb{L}^2} := (u_1, v_1)_{{\rm L}^2} + (u_2, v_2)_{{\rm L}^2},\\[1mm]
&\qquad U=(u_1, u_2)^{\rm T},\quad V=(v_1, v_2)^{\rm T} \in \mathbb{L}^2(\Omega),\\[1mm]
&\mathbb{L}^r(\Omega) := {\rm L}^r(\Omega) \times {\rm L}^r(\Omega),\quad |U|_{\mathbb{L}^r}^r := |u_1|_{{\rm L}^r}^r + |u_2|_{{\rm L}^r}^r\quad\ U \in \mathbb{L}^r(\Omega)\ (1\leq r < \infty),\\[1mm]
&\mathbb{H}^1_0(\Omega) := {\rm H}^1_0(\Omega) \times {\rm H}^1_0(\Omega),\
(U, V)_{\mathbb{H}^1_0} := (u_1, v_1)_{{\rm H}^1_0} + (u_2, v_2)_{{\rm H}^1_0}\ \ U, V \in \mathbb{H}^1_0(\Omega),\\
&\mathbb{H}^2(\Omega) := {\rm H}^2(\Omega) \times {\rm H}^2(\Omega).
\end{aligned}
\]
We use the differential symbols to indicate differential operators which act on each component of \({\mathbb{H}^1_0}(\Omega)\)-elements:
\[
\begin{aligned}
& D_i = \frac{\partial}{\partial x_i}: \mathbb{H}^1_0(\Omega) \rightarrow \mathbb{L}^2(\Omega),\\
&D_i U = (D_i u_1, D_i u_2)^{\rm T} \in \mathbb{L}^2(\Omega) \ (i=1, \cdots, N),\\[2mm]
& \nabla = \left(\frac{\partial}{\partial x_1}, \cdots, \frac{\partial}{\partial x_N}\right): \mathbb{H}^1_0(\Omega) \rightarrow ({\rm L}^2(\Omega))^{2 N},\\
&\nabla U=(\nabla u_1, \nabla u_2)^T \in ({\rm L}^2(\Omega))^{2 N}.
\end{aligned}
\]

We further define, for \(U=(u_1, u_2)^{\rm T},\ V= (v_1, v_2)^{\rm T},\ W = (w_1, w_2)^{\rm T}\),
\[
\begin{aligned}
&U(x) \cdot \nabla V(x) := u_1(x) \nabla v_1(x) + u_2(x) \nabla v_2(x) \in \mathbb{R}^N,\\[2mm]
&( U(x) \cdot \nabla V(x) ) W(x) := ( u_1(x) ~\! w_1(x) \nabla v_1(x), u_2(x) w_2(x) \nabla v_2(x) )^{\rm T} \in \mathbb{R}^{2N},\\[2mm]
&(\nabla U(x) \cdot \nabla V(x)) := \nabla u_1(x) \cdot \nabla v_1(x) + \nabla u_2(x) \cdot \nabla v_2(x) \in \mathbb{R}^1,\\[2mm]
&|\nabla U(x)| := \left(|\nabla u_1(x)|^2_{\mathbb{R}^N} + |\nabla u_2(x)|^2_{\mathbb{R}^N} \right)^{1/2}.
\end{aligned}
\]

In addition, \(\mathcal{H}^S\) denotes the space of functions with values in \(\mathbb{L}^2(\Omega)\) defined on \([0, S]\) (\(S > 0\)), which is a Hilbert space with the following inner product and norm.
\[
\begin{aligned}
&\mathcal{H}^S := {\rm L}^2(0, S; \mathbb{L}^2(\Omega)) \ni U(t), V(t),\\
&\quad\mbox{with inner product:}\ (U, V)_{\mathcal{H}^S} = \int_0^S (U, V)_{\mathbb{L}^2}^2 dt,\\
&\quad\mbox{and norm:}\ \|U\|_{\mathcal{H}^S}^2 = (U, U)_{\mathcal{H}^S}.
\end{aligned}
\]

As the realization in \(\mathbb{R}^2\) of the imaginary unit \(i\) in \(\mathbb{C}\), we introduce the following matrix \(I\), which is a linear isometry on \(\mathbb{R}^2\):
\[
I = \begin{pmatrix}
0 & -1\\ 1 & 0
\end{pmatrix}.
\]
We abuse \(I\) for the realization of \(I\) in \(\mathbb{L}^2(\Omega)\), i.e., \(I U = ( - u_2, u_1 )^{\rm T}\) for all \(U = (u_1, u_2)^{\rm T} \in \mathbb{L}^2(\Omega)\).

Then \(I\) satisfies the following properties (see \cite{KOS1}):

\begin{enumerate}
\item Skew-symmetric property:
\begin{equation}
\label{skew-symmetric_property}
(IU \cdot V)_{\mathbb{R}^2} = - (U \cdot IV)_{\mathbb{R}^2}; \hspace{4mm}
(IU \cdot U)_{\mathbb{R}^2} = 0 \hspace{4mm}
\mbox{for each}\ U, V \in \mathbb{R}^2.
\end{equation}

\item Commutative property with the differential operator \(D_i = \frac{\partial}{\partial x_i}\):
\begin{equation}
\label{commutative_property}
I D_i = D_i I:\mathbb{H}^1_0 \rightarrow \mathbb{L}^2\ (i=1, \cdots, N).
\end{equation}

\end{enumerate}

Let \({\rm H}\) be a Hilbert space and denote by \(\Phi({\rm H})\) the set of all lower semi-continuous convex function \(\phi\) from \({\rm H}\) into \((-\infty, +\infty]\) such that the effective domain of \(\phi\) given by \({\rm D}(\phi) := \{u \in {\rm H}\mid \ \phi(u) < +\infty \}\) is not empty.
Then for \(\phi \in \Phi({\rm H})\), the subdifferential of \(\phi\) at \(u \in {\rm D}(\phi)\) is defined by
\[
\partial \phi(u) := \{w \in {\rm H}\mid  (w, v - u)_{\rm H} \leq \phi(v)-\phi(u) \hspace{2mm} \mbox{for all}\ v \in {\rm H}\},
\]
which is a possibly multivalued maximal monotone operator with domain\\
\({\rm D}(\partial \phi) = \{u \in {\rm H}\mid  \partial\phi(u) \neq \emptyset\}\).
However for the discussion below, we have only to consider the case where \(\partial \phi\) is single valued.

We define functionals \(\varphi, \ \psi_r:\mathbb{L}^2(\Omega) \rightarrow [0, +\infty]\) (\(r\geq2\)) by
\begin{align}
\label{varphi}
&\varphi(U) :=
\left\{
\begin{aligned}
&\frac{1}{2} \displaystyle\int_\Omega |\nabla U(x)|^2 dx
&&\mbox{if}\ U \in \mathbb{H}^1_0(\Omega),\\[3mm]
&+ \infty &&\mbox{if}\ U \in \mathbb{L}^2(\Omega)\setminus\mathbb{H}^1_0(\Omega),
\end{aligned}
\right.
\\[2mm]
\label{psi}
&\psi_r(U) :=
\left\{
\begin{aligned}
&\frac{1}{r} \displaystyle\int_\Omega |U(x)|_{\mathbb{R}^2}^r dx
&&\mbox{if}\ U \in \mathbb{L}^r(\Omega) \cap \mathbb{L}^2(\Omega),\\[3mm]
&+\infty &&\mbox{if}\ U \in \mathbb{L}^2(\Omega)\setminus\mathbb{L}^r(\Omega).
\end{aligned}
\right.
\end{align}
Then it is easy to see that \(\varphi, \psi_r \in \Phi(\mathbb{L}^2(\Omega))\) and their subdifferentials are given by
\begin{align}
\label{delvaphi}
&\begin{aligned}[t]
&\partial \varphi(U)=-\Delta U\ \mbox{with} \ {\rm D}( \partial \varphi) = \mathbb{H}^1_0(\Omega)\cap\mathbb{H}^2(\Omega),\\[2mm]
\end{aligned}\\
\label{delpsi}
&\partial \psi_r(U) = |U|_{\mathbb{R}^2}^{r-2}U=|U|^{r-2}U\ {\rm with} \ {\rm D}( \partial \psi_r) = \mathbb{L}^{2(r-1)}(\Omega) \cap \mathbb{L}^2(\Omega).
\end{align}
Furthermore for any $\mu>0$, we can define the Yosida approximations
 \(\partial \varphi_\mu,\ \partial \psi_{r,\mu}\) of \(\partial \varphi,\ \partial \psi_r\) by
\begin{align}
\label{Yosida:varphi}
&\partial \varphi_\mu(U) := \frac{1}{\mu}(U - J_\mu^{\partial \varphi}U) 
= \partial \varphi(J_\mu^{\partial \varphi} U), 
\quad J_\mu^{\partial \varphi} : = ( 1 + \mu \partial \varphi)^{-1},
\\[2mm]
\label{Yosida:psi}
&\partial \psi_{r,\mu}(U) := \frac{1}{\mu} (U - J_\mu^{\partial \psi_r} U) 
= \partial \psi_r( J_\mu^{\partial \psi_r} U ), 
\quad  J_\mu^{\partial \psi_r} : = ( 1 + \mu \partial \psi_r)^{-1}.
\end{align}
The second identity holds since \(\partial\varphi\) and \(\partial\psi_r\) are single-valued.

Then it is well known that \(\partial \varphi_\mu, \ \partial \psi_{r,\mu}\) are Lipschitz continuous on \(\mathbb{L}^2(\Omega)\) and satisfies the following properties (see \cite{B1}, \cite{B2}):
\begin{align}
\label{asd}
\psi_r(J_\mu^{\partial\psi_r}U)&\leq\psi_{r,\mu}(U) \leq \psi_r(U),\\
\label{as}
|\partial\psi_{r,\mu}(U)|_{{\rm L}^2}&=|\partial\psi_{r}(J_\mu^{\partial\psi}U)|_{{\rm L}^2}\leq|\partial\psi_r(U)|_{\mathbb{L}^2}\quad\forall\ U \in {\rm D}(\partial\psi_r),\ \forall \mu > 0,
\end{align}
where \(\psi_{r,\mu}\) is the Moreau-Yosida regularization of \(\psi_r\) given by the following formula:
\[
\psi_{r,\mu}(U) = \inf_{V \in \mathbb{L}^2(\Omega)}\left\{
\frac{1}{2\mu}|U-V|_{\mathbb{L}^2}^2+\psi_r(V)
\right\}
=\frac{\mu}{2}|(\partial\psi_r)_\mu(U)|_{\mathbb{L}^2}^2+\psi_r(J_\mu^{\partial\psi}U)\geq0.
\]
Moreover since \(\psi_r(0)=0\), it follows from the definition of subdifferential operators and \eqref{asd} that
\begin{align}\label{asdf}
(\partial\psi_{r,\mu}(U),U)_{\mathbb{L}^2} = (\partial\psi_{r,\mu}(U),U-0)_{\mathbb{L}^2}
\leq \psi_{r,\mu}(U) - \psi_{r,\mu}(0) \leq \psi_r(U).
\end{align}

Here for later use, we prepare some fundamental properties of \(I\) in connection with \(\partial \varphi,\ \partial \psi_r,\  \partial \varphi_\mu,\ \partial \psi_{r,\mu}\).
\begin{Lem2}[(c.f. \cite{KOS1} Lemma 2.1)]
The following angle conditions hold:
\label{Lem:2.1}
\begin{align}
\label{orth:IU}
&(\partial \varphi(U), I U)_{\mathbb{L}^2} = 0\quad 
\forall U \in {\rm D}(\partial \varphi),\quad 
(\partial \psi_r(U), I U)_{\mathbb{L}^2} = 0\quad \forall U \in {\rm D}(\partial \psi_r), 
\\[2mm]
\label{orth:mu:IU}
&(\partial \varphi_\mu(U), I U)_{\mathbb{L}^2} = 0,\quad 
(\partial \psi_{r,\mu}(U), I U)_{\mathbb{L}^2} = 0 \quad 
\forall U \in \mathbb{L}^2(\Omega), 
\\[2mm]
\label{orth:Ipsi}
&\begin{aligned}
&(\partial \psi_q(U), I \partial \psi_r(U))_{\mathbb{L}^2} = 0,\\
&(\partial \psi_q(U), I \partial \psi_{r,\mu}(U))_{\mathbb{L}^2}=0\quad 
\forall U \in {\rm D}(\partial \psi_q)\cap{\rm D}(\partial \psi_r), \forall q,r \geq 2,
\end{aligned}\\
\label{angle}
&(\partial\varphi(U),\partial\psi_r(U))_{\mathbb{L}^2} \geq 0\quad\forall U \in{\rm D}(\partial \varphi)\cap{\rm D}(\partial \psi_r).
\end{align}
\end{Lem2}

\begin{proof}
The first relation in \eqref{orth:Ipsi} is obvious.
So we only give a proof of the second relation in \eqref{orth:Ipsi} here.
Let \(W=J_\mu^{\partial\psi_r}U\), then \(U = W + \mu\partial\psi_r(W)\).
It holds
\[
(\partial \psi_q(U), I \partial \psi_{r,\mu}(U))_{\mathbb{L}^2}
=
(|U|^{q-2}(W + \mu|W|^{r-2}W),I|W|^{r-2}W)_{\mathbb{L}^2}=0.
\]
\end{proof}

We also recall a property of the sum of \(\partial\varphi\) and \(\partial\psi_r\).
\begin{Lem2}[(c.f. \cite{KOS1} Lemma 2.3)]
The operator \(\lambda\partial\varphi(U) + \varepsilon\partial\psi_r(U)\) (\(\varepsilon>0\)) is maximal monotone in \(\mathbb{L}^2(\Omega)\).
\label{Lem:2.3}
Moreover the following relation holds:
\[
\lambda\partial\varphi(U) + \varepsilon\partial\psi_r(U) = \partial(\lambda\varphi+\varepsilon\psi_r)(U).
\]
\end{Lem2}

Then in view of \eqref{delvaphi}, \eqref{delpsi} and the property of \(I\), we can see that (CGL) can be reduced to the following evolution equation:
\[
\tag*{(ACGL)}
\left\{
\begin{aligned}
&\frac{dU}{dt}(t) \!+\! \lambda\partial\varphi(U) \!+\! \alpha I \partial \varphi(U) \!-\! (\kappa+ \beta I) \partial \psi_q(U) \!-\! \gamma U \!=\! F(t),\quad t \in (0,T),\\
&U(0) =U_0,
\end{aligned}
\right.
\]
where \(f(t, x) = f_1(t, x) + i f_2(t, x)\) is identified with \(F(t) = (f_1(t, \cdot), f_2(t, \cdot))^{\rm T} \in \mathbb{L}^2(\Omega)\).

We conclude this section by preparing two lemmas for later use.
The first one is a pointwise estimate for the difference of nonlinear terms:
\begin{Lem}
Let $r \in (2,\infty)$ and put 
\[
d_r
:=
\begin{cases}
\frac{r-1}{2}&\mbox{if}\ 4 \leq r,\\
\frac{3}{2}&\mbox{if}\ 3 < r < 4,\\
1&\mbox{if}\ 2 < r \leq 3.
\end{cases}
\]
Then the following inequality holds.\label{locLip1}
\begin{equation}\label{locLip2}
\begin{aligned}[t]
\left|\left(|U|^{r - 2} u_i - |V|^{r - 2} v_i\right)(x_j - y_j)\right| 
\leq d_r\left(|U|^{r - 2} + |V|^{r - 2}\right) |U-V||X - Y|&\\
\forall i,j=1,2&
\end{aligned}
\end{equation}
for all \(U=(u_1,u_2),V=(v_1,v_2),X=(x_1,x_2),Y=(y_1,y_2)\in \mathbb{R}^2\).
\end{Lem}
This can be proved by the same arguments in the proof of Lemma 5 in \cite{KO2} with obvious modifications (see (6.2) (6.3) and (6.4) in \cite{KO2}).

Next lemma is concerned with the accretivity of the operator \(\partial\psi_q\) in \(\mathbb{L}^r(\Omega)\), namely the following assertion holds:
\begin{Lem}
Let \(V_i = J_\mu^{\partial\psi_q}U_i\) (\(i=1,2\)).\label{ACC}
Then the following inequality holds:
\begin{equation}\label{ACC1}
|V_1-V_2|_{\mathbb{L}^r} \leq |U_1-U_2|_{\mathbb{L}^r}.
\end{equation}
\end{Lem}
\begin{proof}
By the definition of resolvent operators, we have \(U_i = V_i + \mu\partial\psi_q(V_i) = V_i + \mu|V_i|^{q-2}V_i\) (\(i=1,2\)).
Multiplying \(U_1-U_2\) by \(|V_1-V_2|^{r-2}(V_1-V_2)\) and applying H\"older's inequality, we get
\[
\begin{aligned}
&|U_1-U_2|_{\mathbb{L}^r}|V_1-V_2|_{\mathbb{L}^r}^{r-1}\\
&\geq
(U_1-U_2, |V_1-V_2|^{r-2}(V_1-V_2))_{\mathbb{L}^2}\\
&=
|V_1-V_2|_{\mathbb{L}^r}^r
+\mu
(|V_1|^{q-2}V_1-|V_2|^{q-2}V_2, |V_1-V_2|^{r-2}(V_1-V_2))_{\mathbb{L}^2}.
\end{aligned}
\]
Here, to derive \eqref{ACC1}, it suffices to show that
\[
\begin{aligned}
&(|V_1|^{q-2}V_1-|V_2|^{q-2}V_2, |V_1-V_2|^{r-2}(V_1-V_2))_{\mathbb{L}^2}\\
&=
\int_\Omega
|V_1-V_2|^{r-2}
\{|V_1|^q+|V_2|^q-(|V_1|^{q-2}+|V_2|^{q-2})V_1V_2\}
dx\geq0.
\end{aligned}
\]
In fact, by Young's inequality, we get
\[
|V_1|^{q-2}V_1V_2
\leq \left(1-\frac{1}{q}\right)|V_1|^q+\frac{1}{q}|V_2|^q,
\quad
|V_2|^{q-2}V_1V_2
\leq \left(1-\frac{1}{q}\right)|V_2|^q+\frac{1}{q}|V_1|^q.
\]
\end{proof}


\section{Main Results}\label{sec-3}
Our main results are stated as follows.

\setcounter{tmpcnt}{1}
\begin{Thm}[Local well-posedness in general domains]
Let \(\Omega \subset \mathbb{R}^N\) be a general domain of uniformly \({\rm C}^2\)-regular class \label{lwpgd}
, \(F \in \mathcal{H}^T\) and \(2 < q < 2^*\) (subcritical), where
\begin{equation}\label{SobSub}
2^* =
\begin{cases}
+\infty & (N = 1, 2),\\
\frac{2N}{N - 2} & (N \geq 3).
\end{cases}
\end{equation}
Then for all \(U_0 \in \mathbb{H}_0^1(\Omega) = {\rm D}(\varphi)\), there exist \(T_0 \in (0, T]\) and a unique function \(U(t) \in {\rm C}([0, T_0]; \mathbb{H}^1_0(\Omega))\) satisfying:
\begin{enumerate}
\renewcommand{\labelenumi}{(\roman{enumi})}
\item \(U \in {\rm W}^{1, 2}(0, T_0; \mathbb{L}^2(\Omega))\),
\item \(U(t) \in {\rm D}(\partial\varphi) \subset {\rm D}(\partial\psi_q)\) for a.e. \(t \in [0, T_0]\) and satisfies (ACGL) for a.e. \(t \in [0, T_0]\),
\item \(\partial\varphi(U(\cdot)), \partial\psi_q(U(\cdot)) \in \mathcal{H}^{T_0}\).
\end{enumerate}
\end{Thm}
Furthermore the following alternative on the maximal existence time of the solution holds:
\begin{Thm}[Alternative]
Let \(T_m\) be the maximal existence time of a solution to (ACGL) satisfying the regularity (i)-(iii) given in Theorem \ref{lwpgd} for all $T_0 \in (0,T_m)$.\label{altgd}
Then the following alternative on \(T_m\) holds:
\begin{itemize}
\item \(T_m = T\) or
\item \(T_m < T\) and \( \lim_{t \uparrow T_m}\left\{|U(t)|_{\mathbb{L}^2}^2+2\varphi(U(t))\right\} = +\infty\).
\end{itemize}
\end{Thm}

In order to formulate the existence of small global solutions for \(F \in \mathcal{H}^T\), let \(\tilde{F}\) be the extension of \(F\) by zero to \((0, +\infty)\).
We set the following notation in order to measure the external force \(F\) in terms of \(\tilde{F}\):
\[
\opnorm{F}^2_2 :=
\sup\left\{  \int_s^{s + 1}|\tilde{F}(t)|_{\mathbb{L}^2}^2 ~\! dt
\mathrel{;} 0 \leq s < +\infty \right\}.
\]
\begin{Thm}[Existence of small global solutions]
Let all assumptions in Theorem \ref{lwpgd} be satisfied and let \(\gamma < 0\).\label{gegd}
Then there exists a sufficiently small number \(r\) independent of \(T\) such that for all \(U_0 \in D(\varphi)\) and \(F \in {\rm L}^2(0, T; \mathbb{L}^2(\Omega))\) with \(\varphi(U_0)+\frac{1}{2}|U_0|_{\mathbb{L}^2}^2 \leq r^2\) and \(\opnorm{F}_2 \leq r^2\), every local solution given in Theorem \ref{lwpgd} can be continued globally up to \([0, T]\).
\end{Thm}

\section{Auxiliary Problems}\label{sec-4}
In this section, we consider the following auxiliary equations and show their global well-posedness.
\[
\tag*{(AE)\(^\varepsilon\)}
\left\{
\begin{aligned}
&\frac{dU}{dt}(t) \!+\! \lambda\partial\varphi(U)\!+\!\alpha I\partial\varphi(U)\!+\!\varepsilon\partial\psi_r(U)\!-\!(\kappa\!+\! \beta I) \partial \psi_q(U) \!-\! \gamma U \!=\! F(t),\ t \!\in\! (0,T),\\
&U(0) = U_0.
\end{aligned}
\right.
\]

\setcounter{tmpcnt}{1}
\begin{Prop}
Let \(\Omega \subset \mathbb{R}^N\) be a general domain of uniformly \({\rm C}^2\)-regular class
, \(F \in \mathcal{H}^T\), \(2 < q < 2^*\), \(\varepsilon>0\) and \(q<r<2^*\).\label{GWP}
Then for all \(U_0 \in \mathbb{H}_0^1(\Omega) = {\rm D}(\varphi)\), there exists a unique function \(U(t) \in {\rm C}([0, T]; \mathbb{H}^1_0(\Omega))\) satisfying:
\begin{enumerate}
\renewcommand{\labelenumi}{(\roman{enumi})}
\item \(U \in {\rm W}^{1, 2}(0, T; \mathbb{L}^2(\Omega))\),
\item \(U(t) \in {\rm D}(\partial\varphi) \subset {\rm D}(\partial\psi_r)\) for a.e. \(t \in [0, T]\) and satisfies (AE)\(^\varepsilon\) for a.e. \(t \in [0, T]\),
\item \(\partial\varphi(U(\cdot)), \partial\psi_r(U(\cdot)) \in \mathcal{H}^{T}\).
\end{enumerate}
\end{Prop}

\begin{proof}
We consider another approximate equation:
\[
\tag*{(AE)\(_\mu^\varepsilon\)}
\left\{
\begin{aligned}
&\frac{dU}{dt}(t) + \lambda\partial\varphi(U)+\alpha I\partial\varphi(U)+\varepsilon\partial\psi_r(U)-(\kappa+ \beta I) \partial \psi_{q,\mu}(U) - \gamma U = F(t),\\[-1mm]
&\hspace{80mm} t \in (0,T),\\
&U(0) = U_0.
\end{aligned}
\right.
\]
Since \(r<2^*\) implies \(\mathrm{D}(\varphi)\subset\mathrm{D}(\psi_r)\) and \(U\mapsto-(\kappa+\beta I)\partial\psi_{q,\mu}(U)-\gamma U\) is Lipschitz continuous in \(\mathbb{L}^2(\Omega)\), by virtue of Lemma 2, there exists a unique global solution \(U_\mu\) of (AE)\(_\mu^\varepsilon\) satisfying (i)-(iii) of Proposition \ref{GWP} (see \cite{B1} and Proposition 5.1 of \cite{KO2}).

To see the convergence of \(U_\mu\) as \(\mu \downarrow 0\), we establish a priori estimates.
For this purpose, we frequently use the following interpolation inequalities
\begin{equation}\label{inter}
\begin{aligned}[t]
&|U|_{\mathbb{L}^q}^q \leq |U|_{\mathbb{L}^r}^{r\theta}|U|_{\mathbb{L}^2}^{2(1-\theta)},\\
&|U|_{\mathbb{L}^{2(q-1)}}^{2(q-1)} \leq |U|_{\mathbb{L}^{2(r-1)}}^{2(r-1)\theta}|U|_{\mathbb{L}^2}^{2(1-\theta)}=|\partial\psi_r(U)|_{\mathbb{L}^2}^{2\theta}|U|_{\mathbb{L}^2}^{2(1-\theta)},\\
&\theta=\frac{q-2}{r-2}.
\end{aligned}
\end{equation}

\begin{Lem}
Let \(U=U_\mu\) be the solution of (AE)\(_\mu^\varepsilon\).\label{1eAEm}
Then there exists \(C_1\) depending only on \(\lambda, \kappa, \beta, \gamma, \varepsilon\), \(T\), \(|U_0|_{\mathbb{L}^2}\) and \(\|F\|_{\mathcal{H}^T}\) but not on \(\mu\) such that
\begin{equation}
\label{1eAEm-1}
\sup_{t \in [0, T]}|U(t)|_{\mathbb{L}^2}^2 + \int_{0}^{T}\varphi(U(t)) dt + \int_0^T\psi_r(U(t))dt \leq C_1.
\end{equation}
\end{Lem}
\begin{proof}
Multiplying (AE)\(_\mu^\varepsilon\) by \(U\) and noting orthogonalities \eqref{orth:IU}, \eqref{orth:mu:IU}, \eqref{asdf} with \(r=q\) and \eqref{inter}, we obtain
\begin{equation}\label{lkjhgfds}
\begin{aligned}
\frac{1}{2}\frac{d}{dt}|U|_{\mathbb{L}^2}^2
+2\lambda\varphi(U)
+r\varepsilon\psi_r(U)
&\leq
\kappa\psi_q(U)
+\gamma_+|U|_{\mathbb{L}^2}^2
+|F|_{\mathbb{L}^2}|U|_{\mathbb{L}^2}\\
&\leq \frac{\kappa}{q}|U|^{r\theta}_{\mathbb{L}^r}|U|^{2(1-\theta)}_{\mathbb{L}^2}
+\gamma_+|U|_{\mathbb{L}^2}^2
+|F|_{\mathbb{L}^2}|U|_{\mathbb{L}^2}\\
&\leq \frac{\varepsilon}{2}|U|^r_{\mathbb{L}^r}
+(C_\varepsilon+\gamma_++1)|U|^2_{\mathbb{L}^2}
+\frac{1}{4}|F|_{\mathbb{L}^2}^2,
\end{aligned}
\end{equation}
where \(C_\varepsilon = (1-\theta)\left\{\frac{\kappa}{q}\left(\frac{2\theta}{\varepsilon }\right)^\theta\right\}^{\frac{1}{1-\theta}}\) and \(\gamma_+ = \max\{0,\gamma\}\).
Applying Gronwall's inequality to \eqref{lkjhgfds}, we obtain \eqref{1eAEm-1}.
\end{proof}

\begin{Lem}
Let \(U=U_\mu\) be the solution of (AE)\(_\mu^\varepsilon\).\label{2eAEm}
Then there exists \(C_2\) depending only on \(\lambda, \kappa, \beta, \gamma,\varepsilon\), \(T\), \(|U_0|_{\mathbb{L}^2}, \varphi(U_0)\) and \(\|F\|_{\mathcal{H}^T}\) but not on \(\mu\) such that
\begin{equation}
\label{2eAEm-1}
\begin{aligned}
&\sup_{t \in [0, T]}\varphi(U(t)) + \sup_{t \in [0, T]}\psi_r(U(t))\\
&+ \int_{0}^{T}|\partial\varphi(U(t))|_{\mathbb{L}^2}^2dt + \int_0^T|\partial\psi_r(U(t))|_{\mathbb{L}^2}^2dt +\int_{0}^{T}\left|\frac{dU}{dt}(t)\right|_{\mathbb{L}^2}^2\!dt \leq C_2.
\end{aligned}
\end{equation}
\end{Lem}
\begin{proof}
Multiplying (AE)\(_\mu^\varepsilon\) by \(\partial\varphi(U)\) and \(\partial\psi_r(U)\), by \eqref{orth:Ipsi}, \eqref{angle} and \eqref{as}, we have
\begin{align}
\label{1qw}
&\frac{d}{dt}\varphi(U(t))+\lambda|\partial\varphi(U)|_{\mathbb{L}^2}^2
\leq
\begin{aligned}[t]
&\sqrt{\kappa^2+\beta^2}|\partial\psi_q(U)|_{\mathbb{L}^2}|\partial\varphi(U)|_{\mathbb{L}^2}
+2\gamma_+\varphi(U)\\
&+|F|_{\mathbb{L}^2}|\partial\varphi(U)|_{\mathbb{L}^2},
\end{aligned}\\
\label{2qw}
&\frac{d}{dt}\psi_r(U(t))+\varepsilon|\partial\psi_r(U)|_{\mathbb{L}^2}^2
\leq
\begin{aligned}[t]
&\kappa|\partial\psi_q(U)|_{\mathbb{L}^2}|\partial\psi_r(U)|_{\mathbb{L}^2}
+r\gamma_+\psi_r(U)\\
&+|F|_{\mathbb{L}^2}|\partial\psi_r(U)|_{\mathbb{L}^2}.
\end{aligned}
\end{align}

Using H\"older's inequality, we obtain
\begin{align}
\label{1qw2}
&\frac{d}{dt}\varphi(U(t))+\frac{\lambda}{2}|\partial\varphi(U)|_{\mathbb{L}^2}^2
\leq
\frac{\kappa^2+\beta^2}{\lambda}|\partial\psi_q(U)|_{\mathbb{L}^2}^2
+2\gamma_+\varphi(U)
+\frac{1}{\lambda}|F|_{\mathbb{L}^2}^2,\\
\label{2qw2}
&\frac{d}{dt}\psi_r(U(t))+\frac{\varepsilon}{2}|\partial\psi_r(U)|_{\mathbb{L}^2}^2
\leq
\frac{\kappa^2}{\varepsilon}|\partial\psi_q(U)|_{\mathbb{L}^2}^2
+r\gamma_+\psi_r(U)
+\frac{1}{\varepsilon}|F|_{\mathbb{L}^2}^2.
\end{align}

We add \eqref{1qw2} to \eqref{2qw2} and apply \eqref{inter} to obtain
\begin{equation}
\begin{aligned}
&\frac{d}{dt}\varphi(U(t))
+\frac{d}{dt}\psi_r(U(t))
+\frac{\lambda}{2}|\partial\varphi(U)|_{\mathbb{L}^2}^2
+\frac{\varepsilon}{2}|\partial\psi_r(U)|_{\mathbb{L}^2}^2\\
&\leq
\left(\frac{\kappa^2+\beta^2}{\lambda}+\frac{\kappa^2}{\varepsilon}\right)|\partial\psi_q(U)|_{\mathbb{L}^2}^2
+2\gamma_+\varphi(U)+r\gamma_+\psi_r(U)
+\left(\frac{1}{\lambda}+\frac{1}{\varepsilon}\right)|F|_{\mathbb{L}^2}^2\\
&\leq
\left(\frac{\kappa^2+\beta^2}{\lambda}+\frac{\kappa^2}{\varepsilon}\right)
|\partial\psi_r(U)|_{\mathbb{L}^2}^{2\theta}|U|_{\mathbb{L}^2}^{2(1-\theta)}
+2\gamma_+\varphi(U)+r\gamma_+\psi_r(U)
+\left(\frac{1}{\lambda}+\frac{1}{\varepsilon}\right)|F|_{\mathbb{L}^2}^2\\
&\leq
\frac{\varepsilon}{4}|\partial\psi_r(U)|_{\mathbb{L}^2}^2
+\tilde{C}_\varepsilon|U|_{\mathbb{L}^2}^2
+2\gamma_+\varphi(U)+r\gamma_+\psi_r(U)
+\left(\frac{1}{\lambda}+\frac{1}{\varepsilon}\right)|F|_{\mathbb{L}^2}^2,
\end{aligned}
\end{equation}
where
\[
\tilde{C}_\varepsilon
=
(1-\theta)
\left\{
\left(\frac{4\theta}{\varepsilon}\right)^\theta
\left(\frac{\kappa^2+\beta^2}{\lambda}+\frac{\kappa^2}{\varepsilon}\right)
\right\}^{\frac{1}{1-\theta}}.
\]
Then \eqref{1eAEm-1} and Gronwall's inequality assure the estimates for the first four terms in \eqref{2eAEm-1}.
Hence the estimate for \(\left|\frac{dU}{dt}\right|_{\mathrm{L}^2(0,T;\mathbb{L}^2(\Omega))}\) follows from the equation.
\end{proof}
Noting \eqref{inter} and the above estimates, we can deduce
\begin{equation}\label{2eAEm-2}
\sup_{t\in[0,T]}\psi_q(U(t))+\int_0^T|\partial\psi_q(U(t))|_{\mathbb{L}^2}^2dt\leq C_2.
\end{equation}

Let \(U_\mu\) and \(U_\nu\) be solutions to
\begin{align}
\tag*{(AE)\(_\mu\)}
&\frac{dU_\mu}{dt}(t) + (\lambda+\alpha I)\partial\varphi(U_\mu)+\varepsilon\partial\psi_r(U_\mu)-(\kappa+ \beta I) \partial \psi_{q,\mu}(U_\mu) - \gamma U_\mu = F(t),\\
\tag*{(AE)\(_\nu\)}
&\frac{dU_\nu}{dt}(t) + (\lambda+\alpha I)\partial\varphi(U_\nu)+\varepsilon\partial\psi_r(U_\nu)-(\kappa+ \beta I) \partial \psi_{q,\nu}(U_\nu) - \gamma U_\nu = F(t)
\end{align}
with initial condition \(U_\mu(0) = U_\nu(0) = U_0\) respectively.

Multiplying (AE)\(_\mu\)\(-\)(AE)\(_\nu\) by \(U_\mu-U_\nu\) and using 
monotonicity of \(I\partial\varphi\) and \(\partial\psi_r\), we obtain
\[
\begin{aligned}
&\frac{1}{2}\frac{d}{dt}|U_\mu-U_\nu|_{\mathbb{L}^2}^2
+2\lambda\varphi(U_\mu-U_\nu)\\
&\leq
\bigl((\kappa+ \beta I)(\partial\psi_{q,\mu}(U_\mu)-\partial\psi_{q,\nu}(U_\nu)), U_\mu-U_\nu\bigr)_{\mathbb{L}^2}
+\gamma_+|U_\mu-U_\nu|_{\mathbb{L}^2}^2.
\end{aligned}
\]
Integrating the above inequality over \([0,T]\) with respect to \(t\), we obtain
\begin{equation}\label{gfd}
\begin{aligned}
&\frac{1}{2}|U_\mu-U_\nu|_{\mathbb{L}^2}^2
+2\lambda \int_0^T\varphi(U_\mu-U_\nu)dt
\\
&\leq
\bigl((\kappa+ \beta I)(\partial\psi_{q,\mu}(U_\mu)-\partial\psi_{q,\nu}(U_\nu)), U_\mu-U_\nu\bigr)_{\mathcal{H}^T}
+\gamma_+\!\int_0^T|U_\mu-U_\nu|_{\mathbb{L}^2}^2dt.
\end{aligned}
\end{equation}

The first term on the right hand side of \eqref{gfd} can be decomposed in the following way:
\begin{equation}\label{gfds}
\begin{aligned}
&\bigl((\kappa+ \beta I)(\partial\psi_{q,\mu}(U_\mu)-\partial\psi_{q,\nu}(U_\nu)), U_\mu-U_\nu\bigr)_{\mathcal{H}^T}\\
&=\int_0^T
\bigl((\kappa+ \beta I)(\partial\psi_q(J_\mu^{\partial\psi_q}U_\mu)-\partial\psi_q(J_\nu^{\partial\psi_q}U_\nu)), U_\mu-U_\nu\bigr)_{\mathbb{L}^2}
dt\\
&=
\begin{aligned}[t]
&\int_0^T
\bigl((\kappa+ \beta I)(\partial\psi_q(J_\mu^{\partial\psi_q}U_\mu)-\partial\psi_q(J_\mu^{\partial\psi_q}U_\nu)), U_\mu-U_\nu\bigr)_{\mathbb{L}^2}
dt\\
&+
\int_0^T
\bigl((\kappa+ \beta I)(\partial\psi_q(J_\mu^{\partial\psi_q}U_\nu)-\partial\psi_q(J_\nu^{\partial\psi_q}U_\nu)), U_\mu-U_\nu\bigr)_{\mathbb{L}^2}
dt.
\end{aligned}
\end{aligned}
\end{equation}

Put \(\tilde{C} := \sqrt{\kappa^2+\beta^2}d_q\).
Then by \eqref{locLip2} and \eqref{ACC1}, we have
\begin{equation}\label{gfdsa}
\begin{aligned}
&\int_0^T
\bigl((\kappa+ \beta I)(\partial\psi_q(J_\mu^{\partial\psi_q}U_\mu)-\partial\psi_q(J_\mu^{\partial\psi_q}U_\nu)), U_\mu-U_\nu\bigr)_{\mathbb{L}^2}
dt\\
&=
\int_0^T
\int_\Omega
(\kappa+ \beta I)(|J_\mu^{\partial\psi_q}U_\mu|^{q-2}J_\mu^{\partial\psi_q}U_\mu-|J_\mu^{\partial\psi_q}U_\nu|^{q-2}J_\mu^{\partial\psi_q}U_\nu), U_\mu-U_\nu\bigr)_{\mathbb{L}^2}
dxdt\\
&\leq
\tilde{C}
\int_0^T
\int_\Omega
\bigl(
|J_\mu^{\partial\psi_q}U_\mu|^{q-2}
+
|J_\mu^{\partial\psi_q}U_\nu|^{q-2}
\bigr)
|J_\mu^{\partial\psi_q}U_\mu
-
J_\mu^{\partial\psi_q}U_\nu|
|U_\mu-U_\nu|
dxdt\\
&\leq
\tilde{C}q^{\frac{q-2}{q}}
\int_0^T
\bigl(
\psi_q(J_\mu^{\partial\psi_q}U_\mu)^{\frac{q-2}{q}}
\!+\!
\psi_q(J_\mu^{\partial\psi_q}U_\nu)^{\frac{q-2}{q}}
\bigr)
|J_\mu^{\partial\psi_q}U_\mu
\!-\!
J_\mu^{\partial\psi_q}U_\nu|_{\mathbb{L}^q}
|U_\mu-U_\nu|_{\mathbb{L}^q}
dt
\\
&\leq
\tilde{C}q^{\frac{q-2}{q}}
\int_0^T
\bigl(
\psi_q(J_\mu^{\partial\psi_q}U_\mu)^{\frac{q-2}{q}}
+
\psi_q(J_\mu^{\partial\psi_q}U_\nu)^{\frac{q-2}{q}}
\bigr)
|U_\mu-U_\nu|_{\mathbb{L}^q}^2
dt\\
&\leq
\lambda
\int_0^T
\varphi(U_\mu-U_\nu)
dt
+
\bar{C}
\int_0^T
|U_\mu-U_\nu|_{\mathbb{L}^2}^2
dt,
\end{aligned}
\end{equation}
where
we used estimates \eqref{2eAEm-2}, \eqref{asd}, the interpolation inequality:
\begin{equation}\label{GNH1}
\begin{aligned}[t]
&|U|_{\mathbb{L}^q}^q
\begin{aligned}[t]
&\leq
C_b^{\frac{q}{2}}
|\nabla U|_{\mathbb{L}^2}^{2\cdot\frac{q-\xi}{q}}
|U|_{\mathbb{L}^2}^{2\cdot\frac{\xi}{q}}\\
&\leq
\frac{\lambda}{2}|\nabla U|_{\mathbb{L}^2}^2
+
\bar{C}|U|_{\mathbb{L}^2}^2,
\end{aligned}\\
&\xi=\frac{2^*-q}{2(N-2)}\in(0,q)\ \mbox{for}\ N\geq 3\ \mbox{and}\ \xi=\frac{1}{2}\ \mbox{for}\ N=1,2,
\end{aligned}
\end{equation}
and Young's inequality and \(\bar{C}\) is a constant depending on \(\lambda,\tilde{C},q,C_b\) and \(\xi\).

As for the second term in \eqref{gfds}, 
again by \eqref{locLip2}, we get
\begin{equation}\label{bvc}
\begin{aligned}
&\int_0^T
\bigl((\kappa+ \beta I)(\partial\psi_q(J_\mu^{\partial\psi_q}U_\nu)-\partial\psi_q(J_\nu^{\partial\psi_q}U_\nu)), U_\mu-U_\nu\bigr)_{\mathbb{L}^2}
dt\\
&=
\int_0^T
\int_\Omega
(\kappa+ \beta I)(|J_\mu^{\partial\psi_q}U_\nu|^{q-2}J_\mu^{\partial\psi_q}U_\nu-|J_\nu^{\partial\psi_q}U_\nu|^{q-2}J_\nu^{\partial\psi_q}U_\nu), U_\mu-U_\nu\bigr)_{\mathbb{L}^2}
dxdt\\
&\leq
\tilde{C}
\int_0^T
\int_\Omega
\bigl(
|J_\mu^{\partial\psi_q}U_\nu|^{q-2}
+
|J_\nu^{\partial\psi_q}U_\nu|^{q-2}
\bigr)
|J_\mu^{\partial\psi_q}U_\nu
-
J_\nu^{\partial\psi_q}U_\nu|
|U_\mu-U_\nu|
dxdt\\
&\leq
\tilde{C}
\sum_{i,j=\mu,\nu}
\int_0^T
\int_\Omega
|J_i^{\partial\psi_q}U_\nu|^{q-2}
|J_\mu^{\partial\psi_q}U_\nu
-
J_\nu^{\partial\psi_q}U_\nu|
|U_j|
dxdt\\
&\leq
\tilde{C}
\sum_{i,j=\mu,\nu}
\int_0^T
|J_i^{\partial\psi_q}U_\nu|_{\mathbb{L}^{2(q-1)}}^{q-2}
|J_\mu^{\partial\psi_q}U_\nu
-
J_\nu^{\partial\psi_q}U_\nu|_{\mathbb{L}^2}
|U_j|_{\mathbb{L}^{2(q-1)}}
dt,
\end{aligned}
\end{equation}
where we note
\[
\frac{q-2}{2(q-1)}+\frac{1}{2}+\frac{1}{2(q-1)}=1.
\]

Let \(V_1=J_\mu^{\partial\psi_q}U_\nu\) and \(V_2=J_\nu^{\partial\psi_q}U_\nu\), then the definition of resolvent operator yields \(U_\nu=V_1+\mu\partial\psi_q(V_1)=V_2+\nu\partial\psi_q(V_2)\), that is
\[
J_\mu^{\partial\psi_q}U_\nu-J_\nu^{\partial\psi_q}U_\nu
=
V_1-V_2
=
\nu\partial\psi_q(V_2)
-
\mu\partial\psi_q(V_1),
\]
whence follows
\begin{equation}\label{poiuyt}
\begin{aligned}
|J_\mu^{\partial\psi_q}U_\nu
-
J_\nu^{\partial\psi_q}U_\nu|_{\mathbb{L}^2}
&\leq
(\mu+\nu)
(
|\partial\psi_q(V_2)|_{\mathbb{L}^2}
+
|\partial\psi_q(V_1)|_{\mathbb{L}^2}
)\\
&=
(\mu+\nu)
(
|\partial\psi_q(J_\nu^{\partial\psi_q}U_\nu)|_{\mathbb{L}^2}
+
|\partial\psi_q(J_\mu^{\partial\psi_q}U_\nu)|_{\mathbb{L}^2}
).
\end{aligned}
\end{equation}
Combining \eqref{bvc} with \eqref{poiuyt}, we have
\begin{equation}\label{LKJ}
\begin{aligned}
&\int_0^T
\bigl((\kappa+ \beta I)(\partial\psi_q(J_\mu^{\partial\psi_q}U_\nu)-\partial\psi_q(J_\nu^{\partial\psi_q}U_\nu)), U_\mu-U_\nu\bigr)_{\mathbb{L}^2}
dt\\
&\leq
(\mu+\nu)\tilde{C}
\sum_{i,j,k=\mu,\nu}
\int_0^T
|J_i^{\partial\psi_q}U_\nu|_{\mathbb{L}^{2(q-1)}}^{q-2}
|\partial\psi_q(J_k^{\partial\psi_q}U_\nu)|_{\mathbb{L}^2}
|U_j|_{\mathbb{L}^{2(q-1)}}
dt\\
&\leq
(\mu+\nu)\tilde{C}
\sum_{i,j,k=\mu,\nu}
\begin{aligned}[t]
&\left\{
\int_0^T
|J_i^{\partial\psi_q}U_\nu|_{\mathbb{L}^{2(q-1)}}^{2(q-1)}
dt
\right\}^{\frac{q-2}{2(q-1)}}\\
&\times\left\{
\int_0^T
|\partial\psi_q(J_k^{\partial\psi_q}U_\nu)|_{\mathbb{L}^2}^2
dt
\right\}^{\frac{1}{2}}
\left\{
\int_0^T
|U_j|_{\mathbb{L}^{2(q-1)}}^{2(q-1)}
dt
\right\}^{\frac{1}{2(q-1)}}
\end{aligned}\\
&\leq(\mu+\nu)\bar{\bar{C}},
\end{aligned}
\end{equation}
where \(\bar{\bar{C}} = 8\tilde{C}C_2\) and we used \eqref{2eAEm-2}, \eqref{as} and the fact that
\[
|\partial\psi_q(U)|_{\mathbb{L}^2}^2
=
\bigl||U|^{q-2}U\bigr|_{\mathbb{L}^2}^2
=
|U|_{\mathbb{L}^{2(q-1)}}^{2(q-1)}.
\]

Thus in view of \eqref{gfd}, \eqref{gfdsa} and \eqref{LKJ}, we obtain
\[
\frac{1}{2}|U_\mu-U_\nu|_{\mathbb{L}^2}^2
+\lambda \int_0^T\varphi(U_\mu-U_\nu)dt
\leq
(\mu+\nu)\bar{\bar{C}}
+(\gamma_++\bar{C})\int_0^T|U_\mu-U_\nu|_{\mathbb{L}^2}^2dt.
\]
Therefore Gronwall's inequality yields that \(\{U_\mu\}_{\mu>0}\) forms a Cauchy net in \(\mathrm{C}([0,T];\mathbb{L}^2(\Omega))\) as \(\mu,\nu\downarrow0\).

By the a priori estimates \eqref{2eAEm-1} and \eqref{2eAEm-2}, we obtain the following convergences of subsequence \(\{U_{\mu_n}\}_{n\in\mathbb{N}}\subset\{U_\mu\}_{\mu>0}\) as \(n\to\infty\):
\begin{align}\label{aaaa}
U_{\mu_n}&\rightarrow U&&\mbox{strongly in}\ {\rm C}([0, T]; \mathbb{L}^2(\Omega)),\\
\frac{dU_{\mu_n}}{dt}&\rightharpoonup \frac{dU}{dt}&&\mbox{weakly in}\ {\rm L}^2(0, T; \mathbb{L}^2(\Omega)),\\
\partial\varphi(U_{\mu_n})&\rightharpoonup \partial\varphi(U)&&\mbox{weakly in}\ {\rm L}^2(0, T; \mathbb{L}^2(\Omega)),\\
\partial\psi_r(U_{\mu_n})&\rightharpoonup \partial\psi_r(U)&&\mbox{weakly in}\ {\rm L}^2(0, T; \mathbb{L}^2(\Omega)),\\
\partial\psi_q(U_{\mu_n})&\rightharpoonup \partial\psi_q(U)&&\mbox{weakly in}\ {\rm L}^2(0, T; \mathbb{L}^2(\Omega)),\\
\partial\psi_{q,{\mu_n}}(U_{\mu_n})&\rightharpoonup \partial\psi_q(U)&&\mbox{weakly in}\ {\rm L}^2(0, T; \mathbb{L}^2(\Omega)),
\end{align}
where we used the demi-closedness of \(\frac{d}{dt}, \partial\varphi, \partial\psi_r, \partial\psi_q\).
We note that \eqref{aaaa} implies \(J_{\mu_n}^{\partial\psi_q}U_{\mu_n} \to U\) strongly in \({\rm L}^2(0,T;\mathbb{L}^2(\Omega))\) (see \cite{KOS1}).

Hence \(U\) is the desired solution of (AE)\(^\varepsilon\) and the uniqueness follows from the fact that \(\{U_\mu\}_{\mu>0}\) forms a Cauchy net in \({\rm C}([0,T];\mathbb{L}^2(\Omega))\).
\end{proof}

\section{Proofs of Theorems \ref{lwpgd} and \ref{altgd}}
In this section we establish local (in time) a priori estimates for solutions \(\{U^\varepsilon\}_{\varepsilon>0}\) of auxiliary equations (AE)\(^\varepsilon\) in order to show the existence of the unique local solution of (ACGL).
In the sequel, we assume \(2<q<r<2^*\).

\begin{Lem}
Let \(U=U^\varepsilon\) be the solution of (AE)\(^\varepsilon\).\label{eAEe}
Then there exist \(C_3\) and \(T_0>0\) depending only on \(\lambda, \kappa, \beta, \gamma\), \(|U_0|_{\mathbb{L}^2}, \varphi(U_0)\) and \(\|F\|_{\mathcal{H}^T}\) but not on \(\varepsilon\) such that
\begin{equation}
\label{eAEe1}
\begin{aligned}
&\sup_{t \in [0, T_0]}|U(t)|_{\mathbb{L}^2}^2
+ \sup_{t \in [0, T_0]}\varphi(U(t)) 
+ \int_{0}^{T_0}|\partial\varphi(U(t))|_{\mathbb{L}^2}^2dt\leq C_3.
\end{aligned}
\end{equation}
\end{Lem}

\begin{proof}
Multiplying (AE)\(^\varepsilon\) by \(U\) and \(\partial\varphi(U)\) and using Young's inequality, \eqref{orth:IU}, \eqref{angle} and \eqref{GNH1}, we obtain the following inequalities respectively:
\begin{align}\label{Poi}
&\begin{aligned}[t]
\frac{1}{2}\frac{d}{dt}|U|_{\mathbb{L}^2}^2
+
2\lambda\varphi(U)
&\leq
(\gamma_++1)|U|_{\mathbb{L}^2}^2
+\frac{1}{4}|F|_{\mathbb{L}^2}^2
+q\kappa\psi_q(U)\\
&\leq
(\gamma_++1)|U|_{\mathbb{L}^2}^2
+|F|_{\mathbb{L}^2}^2
+\kappa
\left\{
\lambda
\varphi(U)
+
\bar{C}
|U|_{\mathbb{L}^2}^2
\right\}^{\frac{q}{2}},
\end{aligned}\\
&\begin{aligned}[t]
\frac{d}{dt}\varphi(U)
+
\frac{\lambda}{2}|\partial\varphi(U)|_{\mathbb{L}^2}^2
&\leq
2\gamma_+\varphi(U)
+\frac{1}{\lambda}|F|_{\mathbb{L}^2}^2
+\frac{\kappa^2+\beta^2}{\lambda}|\partial\psi_q(U)|_{\mathbb{L}^2}^2\\
&\leq
\begin{aligned}[t]
&2\gamma_+\varphi(U)
+\frac{1}{\lambda}|F|_{\mathbb{L}^2}^2\\
&+
\frac{k}{\lambda}(\kappa^2+\beta^2)
\left[
|\partial\varphi(U)|_{\mathbb{L}^2}^{2-\theta}\cdot
(2\varphi(U))^{\frac{2q-4+\theta}{2}}
+
|U|_{\mathbb{L}^2}^{2(q-1)}
\right],
\end{aligned}
\end{aligned}
\end{align}
where we used (in \cite{OS1})
\begin{align}\label{LKJH}
&
|\partial\psi(U)_q(U)|_{\mathbb{L}^2}^2
=
|U|_{\mathbb{L}^{2(q-1)}}^{2(q-1)}
\leq
k
(
|\Delta U|_{\mathbb{L}^2}^{2-\theta}|\nabla U|_{\mathbb{L}^2}^{2q-4+\theta}+|U|_{\mathbb{L}^2}^{2(q-1)}
),
\\
\notag
&
\theta
=
\left\{
\begin{aligned}
&2&&\mbox{if}\ N=1,2\ \mbox{or if}\ N\geq 3\ \mbox{and}\ q\in\left(2,\frac{2N-2}{N-2}\right],
\\
&2q-N(q-2)&&\mbox{if}\ N\geq 3\ \mbox{and}\ \frac{2N-2}{N-2}<q,
\end{aligned}
\right.
\end{align}
where \(k\) is a constant depending on \(q\), \(N\) and \(\Omega\).

Hence, since \(\frac{2N-2}{N-2}<q<2^*\) implies that \(\theta\in(0,2)\), by Young's inequality, there exists \(C_0>0\) such that
\begin{equation}\label{POi}
\frac{d}{dt}\varphi(U)
+
\frac{\lambda}{4}|\partial\varphi(U)|_{\mathbb{L}^2}^2
\leq
2\gamma_+\varphi(U)
+\frac{1}{\lambda}|F|_{\mathbb{L}^2}^2\\
+C_0
\left(
|U|_{\mathbb{L}^2}^{2(q-1)}
+
\varphi(U)^\rho
\right)
\end{equation}
with \(\rho=1+\frac{2(q-2)}{\theta}>q-1\).

We add \eqref{Poi} and \eqref{POi} together to obtain
\begin{equation}\label{POI}
\frac{1}{2}\frac{d}{dt}|U|_{\mathbb{L}^2}^2
+
\frac{d}{dt}\varphi(U)
+
\frac{\lambda}{4}|\partial\varphi(U)|_{\mathbb{L}^2}^2
\leq
\left(1+\frac{1}{\lambda}\right)|F|_{\mathbb{L}^2}^2
+l\left(
|U|_{\mathbb{L}^2}^2+2\varphi(U)\right),
\end{equation}
where \(l(s)=(2(\gamma_++1)s
+
\kappa\left(\frac{\lambda}{2}+2\bar{C}\right)^\frac{q}{2}2s^{\frac{q}{2}}
+
C_0\{(2s)^{2(q-1)}
+
s^\rho\}\) is a non-decreasing function.

Here we recall the following lemma:
\begin{Lem*}[(\^Otani \cite{O5}, p. 360. Lemma 2.2)]
Let \(y(t)\) be a bounded measurable non-negative function on \([0, T]\) and suppose that there exist \(y_0 \geq 0\) and a monotone non-decreasing function \(m(\cdot): [0, +\infty) \to [0, +\infty)\) such that
\begin{equation}
\label{O5Lem22:1}
y(t) \leq y_0 + \int_0^tm(y(s))ds\quad\mbox{a.e.}\ t \in (0, T).
\end{equation}
Then there exists a number \(S = S(y_0, m(\cdot)) \in (0, T]\) such that
\begin{equation}
\label{O5Lem22:2}
y(t) \leq y_0 + 1\quad\mbox{a.e.}\ t \in [0, S].
\end{equation}
\end{Lem*}
We apply the above lemma with \(y(t) = |U(t)|_{\mathbb{L}^2}^2 + 2\varphi(U(t))\), \(y_0=|U_0|_{\mathbb{L}^2}^2 + 2\varphi(U_0)+2\left(1+\frac{1}{\lambda}\right)\|F\|_{\mathcal{H}^T}^2\) and \(m(\cdot)=2l(\cdot)\) so that we obtain \eqref{eAEe1} with
\begin{equation}\label{Szero}
T_0 = S\left(|U_0|_{\mathbb{L}^2}^2 + 2\varphi(U_0)+2\left(1+\frac{1}{\lambda}\right)\|F\|_{\mathcal{H}^T}^2, 2l(\cdot)\right).
\end{equation}
\end{proof}

\begin{proof}[Proof of Theorem \ref{lwpgd}]
By a priori estimate \eqref{eAEe1}, inequality \eqref{LKJH}with \(q\) replaced by \(r\) and assumption \(2 < q < r < 2^*\), the following estimates can be derived as well
\begin{equation}\label{aaa}
\begin{aligned}
&\sup_{t\in[0,T_0]}\psi_q(U(t))
+\sup_{t\in[0,T_0]}\psi_r(U(t))\\
&+\int_{0}^{T_0}|\partial\psi_r(U(t))|_{\mathbb{L}^2}^2dt
+\int_{0}^{T_0}|\partial\psi_q(U(t))|_{\mathbb{L}^2}^2dt
\leq C_3.
\end{aligned}
\end{equation}

Moreover we have the strong convergence of \(\{U^\varepsilon\}_{\varepsilon>0}\) in \({\rm C}([0,T];\mathbb{L}^2(\Omega))\).
Indeed, multiplying the difference of two equations \({\rm (AE)}^\varepsilon-{\rm (AE)}^{\varepsilon'}\) by \(U^\varepsilon - U^{\varepsilon'}\), using the self-adjoint property of \(\partial\varphi\), (\ref{orth:IU}) and by the same argument as in \eqref{gfdsa}
, we get
\begin{equation}
\label{DUeUep}
\begin{aligned}
&\frac{1}{2}\frac{d}{dt}|U^\varepsilon - U^{\varepsilon'}|_{\mathbb{L}^2}^2 + 2\lambda\varphi(U^\varepsilon - U^{\varepsilon'})
+ (\varepsilon\partial\psi_r(U^\varepsilon) - \varepsilon'\partial\psi_r(U^{\varepsilon'}), U^\varepsilon - U^{\varepsilon'})_{\mathbb{L}^2}\\
&\leq \gamma_+|U^\varepsilon - U^{\varepsilon'}|_{\mathbb{L}^2}^2
+ ((\kappa + I\beta)(\partial\psi_q(U^\varepsilon) - \partial\psi_q(U^{\varepsilon'})), U^\varepsilon - U^{\varepsilon'})_{\mathbb{L}^2}\\
&\leq
\gamma_+|U^\varepsilon - U^{\varepsilon'}|_{\mathbb{L}^2}^2 + \tilde{\tilde{C}}(\psi_q(U^\varepsilon)^{\frac{q - 2}{q}}
+ \psi_q(U^{\varepsilon'})^{\frac{q - 2}{q}})|U^\varepsilon - U^{\varepsilon'}|_{\mathbb{L}^q}^2
+\lambda\varphi(U^\varepsilon-U^{\varepsilon'}),
\end{aligned}
\end{equation}
where the constant \(\tilde{\tilde{C}}=\tilde{C}q^{\frac{q-2}{q}}\) depends only on \(q, \kappa, \beta\).

We here assume \(\varepsilon < \varepsilon'\) without loss of generality.
By monotonicity of \(\partial\psi_r\) and the definition of subdifferential operators, we obtain
\begin{equation}
\label{eeprime}
\begin{aligned}
&(\varepsilon\partial\psi_r(U^\varepsilon) - \varepsilon'\partial\psi_r(U^{\varepsilon'}), U^\varepsilon - U^{\varepsilon'})_{\mathbb{L}^2}\\
&= \varepsilon(\partial\psi_r(U^\varepsilon) - \partial\psi_r(U^{\varepsilon'}), U^\varepsilon - U^{\varepsilon'})_{\mathbb{L}^2} + (\varepsilon - \varepsilon')(\partial\psi_r(U^{\varepsilon'}), U^\varepsilon - U^{\varepsilon'})_{\mathbb{L}^2}\\
&\geq (\varepsilon - \varepsilon')(\psi_r(U^\varepsilon) - \psi_r(U^{\varepsilon'})).
\end{aligned}
\end{equation}
Then in view of \eqref{aaa}, \eqref{DUeUep} and \eqref{eeprime}, we obtain
\begin{equation}
\label{DUeUep1}
\begin{aligned}
&\frac{1}{2}\frac{d}{dt}|U^\varepsilon - U^{\varepsilon'}|_{\mathbb{L}^2}^2 + \lambda\varphi(U^\varepsilon - U^{\varepsilon'})\\
&\leq
\begin{aligned}[t]
&\gamma_+|U^\varepsilon - U^{\varepsilon'}|_{\mathbb{L}^2}^2 + \tilde{\tilde{C}}(\psi_q(U^\varepsilon)^{\frac{q - 2}{q}} + \psi_q(U^{\varepsilon'})^{\frac{q - 2}{q}})|U^\varepsilon - U^{\varepsilon'}|_{\mathbb{L}^q}^2\\
&+ (\varepsilon' - \varepsilon)(\psi_r(U^\varepsilon) - \psi_r(U^{\varepsilon'})).
\end{aligned}\\
&\leq
\gamma_+|U^\varepsilon - U^{\varepsilon'}|_{\mathbb{L}^2}^2 + 2\tilde{\tilde{C}}C_3^{\frac{q - 2}{q}}
|U^\varepsilon - U^{\varepsilon'}|_{\mathbb{L}^q}^2
+ (\varepsilon' - \varepsilon)C_3.
\end{aligned}
\end{equation}
Applying \eqref{GNH1} and Young's inequality to \eqref{DUeUep1}, we see that there exists a constant \(C_4\) such that
\[
\begin{aligned}
&\frac{1}{2}\frac{d}{dt}|U^\varepsilon - U^{\varepsilon'}|_{\mathbb{L}^2}^2 + \lambda\varphi(U^\varepsilon - U^{\varepsilon'})\\
&\leq
\gamma_+|U^\varepsilon - U^{\varepsilon'}|_{\mathbb{L}^2}^2 + \frac{\lambda}{2}
|\nabla(U^\varepsilon - U^{\varepsilon'})|_{\mathbb{L}^2}^2
+C_4|U^\varepsilon - U^{\varepsilon'}|_{\mathbb{L}^2}^2
+ (\varepsilon' - \varepsilon)C_3.
\end{aligned}
\]
Thus by Gronwall's inequality, we can conclude that \(\{U^\varepsilon\}_{\varepsilon > 0}\) forms a Cauchy net in \(\mathrm{C}([0,T_0];\mathbb{L}^2(\Omega))\).

By a priori estimates \eqref{eAEe1} and \eqref{aaa}, we can extract a subsequence \(\{U^{\varepsilon_n}\}_{n\in\mathbb{N}}\subset\{U^\varepsilon\}_{\varepsilon>0}\) such that:
\begin{align}
U^{\varepsilon_n}&\rightarrow U&&\mbox{strongly in}\ {\rm C}([0, T_0]; \mathbb{L}^2(\Omega)),\\
\frac{dU^{\varepsilon_n}}{dt}&\rightharpoonup \frac{dU}{dt}&&\mbox{weakly in}\ {\rm L}^2(0, T_0; \mathbb{L}^2(\Omega)),\\
\partial\varphi(U^{\varepsilon_n})&\rightharpoonup \partial\varphi(U)&&\mbox{weakly in}\ {\rm L}^2(0, T_0; \mathbb{L}^2(\Omega)),\\
\varepsilon_n\partial\psi_r(U^{\varepsilon_n})&\rightarrow 0&&\mbox{strongly in}\ {\rm L}^2(0, T_0; \mathbb{L}^2(\Omega)),\\
\partial\psi_q(U^{\varepsilon_n})&\rightharpoonup \partial\psi_q(U)&&\mbox{weakly in}\ {\rm L}^2(0, T_0; \mathbb{L}^2(\Omega)),
\end{align}
where we used the demi-closedness of \(\frac{d}{dt}, \partial\varphi, \partial\psi_q\).
Then wee see that \(U\) is the desired solution of (ACGL).

The uniqueness part follows from the fact that \(\{U^\varepsilon\}_{\varepsilon>0}\) forms a Cauchy net.

\end{proof}
\begin{proof}[Proof of Theorem \ref{altgd}]
We shall proceed the proof of Theorem \ref{altgd} by contradiction.

Let \(T_m<T\) and \(\liminf_{t\uparrow T_m}\left\{|U(t)|_{\mathbb{L}^2}^2+2\varphi(U(t))\right\}<+\infty\).
Then there exists a increasing sequence \(\{t_n\}_{n\in\mathbb{N}}\) and a positive number \(K_0>0\) which is independent of \(n\) such that
\begin{align}
\label{convtn}
&t_n\uparrow T_m\\
\intertext{and}
&|U(t_n)|_{\mathbb{L}^2}^2+2\varphi(U(t_n))\leq K_0\quad\mbox{for all}\ n\in\mathbb{N}.
\end{align}

We repeat the same argument as in the proof of Theorem \ref{lwpgd} with \(t_0\) replaced by \(t_n\) (\(n\in\mathbb{N}\)).
Then by \eqref{Szero}, \(U(t)\) can be continued up to \(t_n+T_0\) as a solution of (ACGL), where
\[
0<T_0=S\left(K_0+2\left(1+\frac{1}{\lambda}\right)\|F\|_{\mathcal{H}^T}^2,2l(\cdot)\right),
\]
which is independent of \(n\).
By \eqref{convtn}, we can take a sufficiently large number \(n\in\mathbb{N}\) such that \(T_m<t_n+T_0\), which contradicts the definition of \(T_m\).
\end{proof}

\section{Proof of Theorem \ref{gegd}}
First we recall the following lemma.
\begin{Lem2}[(c.f. \cite{KO2} Lemma 7)]
   Let \(f(t) \in {\rm L}^1(0, T)\) and \(j(t)\) be an absolutely \label{Gtyineq}
     continuous positive function on \([0, S]\) with \(0 < S \leq T\) such that
\begin{equation}\label{7.5}
    \frac{d}{dt} j(t) + \delta j(t) \leq K |f(t)|   \quad \mbox{a.e.}\ t \in [0, S],
\end{equation}
where \(\delta > 0\) and \(K > 0\). Then we have
\begin{equation}\label{7.6}
  \begin{aligned}
    & j(t) \leq j(0) ~\! e^{-\delta t} 
       + \frac{K}{1 - e^{-\delta}} ~\! \opnorm{f}_1 
           \quad  \forall t \in [0, S],
\\
    & \opnorm{f}_1 = \sup\left\{\int_S^{S + 1}|\tilde{f}(t)|dt 
        \mathrel{;} 0 \leq S < \infty\right\},
\end{aligned}
\end{equation}
where \(\tilde{f}\) is the zero extension of \(f\) to \([0, \infty)\).
\end{Lem2}

Hence in order to give an estimate for the real parts of equation (CGL) from below, we prepare the following lemma.
\begin{Lem}
Let all assumptions in Theorem \ref{gegd} be satisfied. \label{coergd}
Then there exists \(\varepsilon_0 > 0\) and \(\delta > 0\) such that for all \(U \in D(\varphi) = \mathbb{H}_0^1(\Omega)\) satisfying \(\frac{1}{2}|U|_{\mathbb{L}^2}^2+\varphi(U)\leq \varepsilon_0\), the following estimate holds.
\begin{equation}
\label{coer1}
(\lambda \partial \varphi U - \kappa \partial \psi_q U - \gamma U, U)_{\mathbb{L}^2} \geq \delta\left( \frac{1}{2}|U|_{\mathbb{L}^2}^2+\varphi(U)\right).
\end{equation}
\end{Lem}
\begin{proof}
We multiply \(\lambda \partial \varphi(U) - \kappa\partial\psi_q(U) - \gamma U\) by \(U\).
Then we get by (\ref{GNH1})
\begin{equation}
\label{ReUgd}
\begin{aligned}
&(\lambda\partial\varphi(U) - \kappa\partial\psi_q(U) - \gamma U, U)_{\mathbb{L}^2}\\
&= 2\lambda\varphi(U) - q\kappa\psi_q(U) - \gamma |U|_{\mathbb{L}^2}^2\\
&\geq 2\delta\left(\frac{1}{2}|U|_{\mathbb{L}^2}^2+\varphi(U)\right)-\kappa C_b^\frac{q}{2}(|U|_{\mathbb{L}^2}^2+2\varphi(U))^\frac{q}{2}\\
&\geq 
\left(2\delta-\kappa C_b^\frac{q}{2} 2^\frac{q}{2} \left(\frac{1}{2}|U|_{\mathbb{L}^2}^2+\varphi(U)\right)^\frac{q-2}{2}\right)
\left(\frac{1}{2}|U|_{\mathbb{L}^2}^2+\varphi(U)\right),
\end{aligned}
\end{equation}
where \(\delta=\min\{\lambda,|\gamma|\}>0\).
Then choosing \(\varepsilon_0=\left(\frac{\delta}{\kappa C_b^\frac{q}{2}2^\frac{q}{2}}\right)^{\frac{2}{q-2}}\), we obtain \eqref{coer1}.
\end{proof}

With the aid of Lemma \ref{coergd}, we ca derive the global boundedness of \(\frac{1}{2}|U(t)|_{\mathbb{L}^2}^2+\varphi(U(t))\) for small initial data.
\begin{Lem}
Let all assumptions in Theorem \ref{gegd} be satisfied. \label{gbddgd}
Then there exist \(\varepsilon_1>0\) and \(L>0\) independent of \(T\) such that for any \(r\in(0,\varepsilon_1)\), if \(\frac{1}{2}|U_0|_{\mathbb{L}^2}^2 + \varphi(U_0) \leq r^2\) and \(\opnorm{F}_2 \leq r\), then the corresponding solution \(U(t)\) on \([0, S]\), \(0 < S \leq T\) satisfies
\begin{equation}
\label{gbddgd1}
\frac{1}{2}|U(t)|_{\mathbb{L}^2}^2 + \varphi(U(t)) < Lr^2\quad \forall t \in [0, S].
\end{equation}
\end{Lem}
\begin{proof}
We fix \(L\) and \(\varepsilon_1\) by
\begin{align}
    L  & = \left[2 + L_1^2+\frac{1}{1 - e^{-2|\gamma|}}\left\{\frac{1}{\lambda}+C_0(L_1^2+L_2)\right\}\right]
\intertext{(\(C_0\) is the constant appering in \eqref{POi}),}
\notag 
    L_2 & = \frac{1}{\delta} \left(L_1 + \frac{1}{2}L_1^2 \right), \quad L_1 
           = 1 + \frac{1}{1 - e^{- \frac{\delta}{2}}},
\\[1mm]
\label{7.11}
   \varepsilon_1 & = \frac{\varepsilon_0}{L} 
          \quad  \mbox{(\(\varepsilon_0\) is the number appearing in Lemma \ref{coergd}).}
\end{align}

Then we claim that (\ref{gbddgd1}) holds true for all \(t \in [0, S]\).
Suppose that this is not the case, then by the continuity of \(\frac{1}{2}|U(t)|_{\mathbb{L}^2}^2+\varphi(U(t))\), there exists \(t_1 \in (0, S)\) such that
\begin{equation}\label{6.8}
       \frac{1}{2}|U(t)|_{\mathbb{L}^2}^2+\varphi(U(t)) < Lr^2 \quad   \forall t \in [0, t_1)  
            \ \  \mbox{and} \ \ \frac{1}{2}|U(t_1)|_{\mathbb{L}^2}^2+\varphi(U(t_1)) = L r^2.
\end{equation}
We are going to show that this leads to a contradiction.
 We first multiply (ACGL) by \(U(t)\) for \(t \in [0, t_1]\).
Then since \(\frac{1}{2}|U(t)|_{\mathbb{L}^2}^2 +\varphi(U(t)) \leq Lr^2 \leq \varepsilon_0\) for all \(t \in [0, t_1]\), 
  Lemma \ref{coergd} and \eqref{orth:IU} gives
\begin{align}\label{7.13}
    & \frac{1}{2}\frac{d}{dt}|U(t)|_{\mathbb{L}^2}^2 
        + \delta\left(\frac{1}{2}|U(t)|_{\mathbb{L}^2}^2+\varphi(U(t))\right) \leq |F(t)|_{\mathbb{L}^2}|U(t)|_{\mathbb{L}^2} 
          & \forall t \in [0, t_1], 
\end{align}
Hence, by (\ref{7.13}) and Lemma \ref{Gtyineq}, we get
\begin{equation}\label{6.10}
    \sup_{0 \leq t \leq t_1}|U(t)|_{\mathbb{L}^2} 
         \leq \left(1 
                 + \frac{1}{1 - e^{-\frac{\delta}{2}}}\right)\varepsilon_0 = L_1 r,
\end{equation}
where we used the fact that 
  \(|U(0)|_{\mathbb{L}^2} \leq \varepsilon_0\) 
                                and \( \opnorm{|F(t)|_{\mathbb{L}^2}}_1 = \opnorm{F}_2 \leq r\).

Hence the integration of (\ref{7.13}) over \((t, t+1)\) gives
\begin{equation}\label{7.16}
    \sup_{0 \leq t < \infty} \int_t^{t + 1} \tilde{\varphi}(U(\tau))d\tau 
              \leq \frac{1}{\delta} \left( L_1 r^2 + \frac{1}{2} ~\! L_1^2 ~\! r^2 \right) = L_2 ~\! r^2,
\end{equation}
where \(\tilde{\varphi}(U(\cdot))\) is the zero extension of \(\varphi(U(\cdot))\) to \([0, \infty)\).

By the same argument as for \eqref{POi}, we have
\begin{equation}\label{7.18}
\begin{aligned}
&\frac{d}{dt} \varphi(U) + 2|\gamma| \varphi(U) +\frac{\lambda}{4}|\partial\varphi(U)|_{\mathbb{L}^2}^2\\
&\leq \frac{1}{\lambda} |F|_{\mathbb{L}^2}^2 +C_0\left(\frac{1}{2}|U|_{\mathbb{L}^2}^{2(q-1)}+\varphi(U)^\rho\right)
              \quad \forall t \in [0, t_1].
\end{aligned}
\end{equation}
Without loss of generality, we can take \(Lr \leq L\varepsilon_1 = \varepsilon_0 \leq 1\).
Then since \(\rho > 1\), in view of \eqref{6.10}, \eqref{7.16} and Lemma \ref{Gtyineq}, we integrate \eqref{7.18} on \([\tilde{t}_1-1,t_1]\) with \(\tilde{t}_1=\max\{1,t_1\}\) to obtain
\begin{equation}
\frac{1}{2}
|U(t_1)|_{\mathbb{L}^2}^2
+
\varphi(U(t_1)) 
\leq \left[1 + \frac{1}{2}L_1^2+ \frac{1}{1 - e^{-2|\gamma|}}\left\{\frac{1}{\lambda}+C_0(L_1^2+L_2)\right\}\right]r^2 < Lr^2,
\end{equation}
which together with \eqref{6.10} contradicts \eqref{6.8}.
\end{proof}

\begin{proof}[Proof of Theorem \ref{gegd}]
Theorem \ref{gegd} is a direct consequence of the uniform boundedness of \(|U|_{\mathbb{L}^2}^2 + 2\varphi(U)\) based on Lemma \ref{gbddgd} and Theorem \ref{altgd}.
\end{proof}


\end{document}